\documentclass{amsart}

\usepackage[utf8]{inputenc}
\usepackage[T1]{fontenc}
\usepackage{amsmath,amsfonts,amssymb,amsthm}
\usepackage{mathrsfs}
\usepackage{graphicx}
\usepackage{graphics}
\usepackage{lscape}
\usepackage{mathptmx}

\newcommand{\E}{{\mathbb{E}}}

\newcommand{\Prob}{{\mathbb{P}}}
\newcommand{\1}{{\mathbf{1}}}
\let\mathcal\mathscr
\let\geq\geqslant
\let\leq\leqslant
\renewcommand\le{\leq}
\renewcommand\ge{\geq}

\def\ee{\mathrm{e}}
\def\dd{\mathrm{d}}

\theoremstyle{plain}
\newtheorem{theo}{Theorem}[section]

\newtheorem{defi}[theo]{Definition}
\newtheorem{prop}[theo]{Proposition}
\newtheorem{lemm}[theo]{Lemma}

\newtheorem{nota}[theo]{Notation}

\newtheorem{rema}{Remark}

\def\tte{t_{\mathrm{e}}}
\def\tti{t_{\mathrm{i}}}
\def\gte{T_{\mathrm{e}}}
\def\gti{T_{\mathrm{i}}}
\def\sse{S_{\mathrm{e}}}
\def\ssi{S_{\mathrm{i}}}
\def\gtt{\mathfrak{T}}
\def\caln{\mathcal{N}}

\def\nn{\mathfrak{n}}
\def\NN{\mathfrak{N}}

\begin{document}

\title{Priors for the Bayesian star paradox}

\author{Mikael Falconnet}

\address{Université Joseph Fourier Grenoble 1\\
Institut Fourier UMR 5582 UJF-CNRS\\
100 rue des Maths, BP 74\\
38402 Saint Martin d'Hères\\
France }

\date{\today}

\begin{abstract}
  We show that the Bayesian star paradox, first proved mathematically
  by Steel and Matsen for a specific class of prior distributions,
  occurs in a wider context including less regular, possibly
  discontinuous, prior distributions. 
\end{abstract}



\maketitle

\section*{Introduction}

In phylogenetics, a
particular resolved tree can be highly supported even when the data is
generated by an unresolved star tree. This unfortunate aspect of the
Bayesian approach to phylogeny 
reconstruction is called the \textit{star paradox\/}. 
Recent studies highlight that the paradox can occur in the simplest
nontrivial setting, namely
for an unresolved rooted tree on three taxa and two states, see Yang
and Rannala~\cite{yang:branch} and Lewis et al.~\cite{lewis:poly}.  
Kolaczkowski and Thornton~\cite{kola:para} presented some
simulations and suggested that artifactual high posteriors for a
particular resolved tree might disappear for very long
sequences. Previous simulations in~\cite{yang:branch} were plagued by
numerical problems, which left unknown the nature of the limiting
distribution on posterior probabilities. For an introduction to the
Bayesian approach to phylogeny reconstruction we refer to chapter~5 of
Yang~\cite{yang:CME}.

The statistical question which supports the star paradox is whether
the Bayesian posterior distribution of the resolutions of a star tree
becomes uniform when the length of the sequence tends to infinity,
that is, in the case of three taxa and two states, whether the posterior distribution
of each resolution converges to $1/3$.  In a recent paper, Steel and
Matsen \cite{steel:baypar} disprove this, thus ruining Kolaczkowski
and Thornton's hope, for a specific class of branch length priors
which they call \textit{tame\/}. More precisely, Steel and Matsen show
that, for 
every tame prior and every fixed $\varepsilon>0$, the posterior
probability of any of the three possible trees stays above
$1-\varepsilon$ with non vanishing probability when the length of the
sequence goes to infinity. This result was recognized by
Yang~\cite{yang:para} and reinforced by theoretical results on the 
posterior probabilities by Susko~\cite{susko:para}.  

Our main result is that Steel and Matsen's conclusion holds for a
wider class of priors, possibly highly irregular, which we call
\textit{tempered\/}.  Recall that Steel and Matsen consider smooth
priors whose densities satisfy some regularity conditions.

The paper is organized as follows.
In Section~\ref{sect:bayfra}, we describe the Bayesian framework of
the star paradox. In Section~\ref{sect:main}, we define the class of
tempered priors on the branch lengths and we state our main result. In
Section~\ref{sect:sm.etendu}, we state an extension of a
technical lemma due to Steel and Matsen, which allows us to extend
their result. In Section~\ref{sect:theoproof}, we prove our main
result. Section~\ref{sect:claim} is devoted to the
proofs of intermediate results. In
Appendix~\ref{appe:tame}, we prove that every tame prior, in 
Steel and Matsen's sense, is tempered, in the sense of this paper, and
we provide examples of tempered, but not tame, prior distributions.
 Finally, in Appendix~\ref{appe:main}, we prove the extension of Steel
 and Matsen's technical lemma stated in
Section~\ref{sect:sm.etendu}.

\section{Bayesian framework for rooted trees on three taxa} 
\label{sect:bayfra}

We consider three taxa, encoded by the set $\tau=\{1,2,3\}$, with two
possible states.  Phylogenies on $\tau$ are supported by one of the four
following trees: the star tree $R_0$ on three taxa and, for
every taxon $i$ in $\tau$, the tree $R_i$ such that $i$
is the outlier. Relying on a commonly used notation, this reads as
$$
R_1=(1,(2,3)),\quad
R_2=(2,(1,3)),\quad
R_3=(3,(1,2)).
$$
The phylogeny based on $R_0$ is specified by the common length of its
three branches, denoted by $t$. 
For each $i$ in $\tau$, the phylogeny based on $R_i$ is specified by
a pair of branch lengths $(\tte,\tti)$, where $\tte$ denotes the
external branch length and $\tti$ the internal branch length, 
see figure~\ref{figu:trees}. 

For instance, in the phylogeny based on $R_1$, the divergence of taxa
$2$ and $3$ 
occurred $\tte$ units of time ago and the divergence of taxon $1$ and a common ancestor of taxa $2$
and $3$ occurred $\tti+\tte$ units of time ago.

\begin{figure}[ht]\label{figu:trees}
\begin{center}
\includegraphics[width=\textwidth]{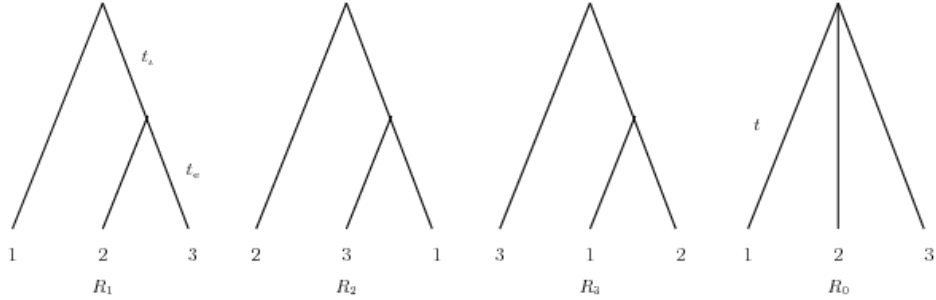}
\end{center}
\caption{The four rooted trees for three species.}
\end{figure}

We assume that the sequences evolve according to a two-state continuous-time
Markov process with equal substitution rates (which we may take to
equal $1$) between the two
character states.  

Four site patterns can occur. The first one, denoted by $s_0$, is such that a given
site coincides in the three taxa. The three others, denoted by $s_i$ with $i$ in $\tau$,
are such that a given site coincide in two taxa and is different
in the third taxon, which is taxon $i$. In other words, if one writes the site patterns in
taxa $1$, $2$ and $3$ in this order and $x$ and $y$ for any two different
characters,
\[
  s_0 = xxx, \quad s_1= yxx, \quad s_2= xyx, \quad \mbox{and} \quad s_3= xxy.
\]
Let $\{s_0,s_1,s_2,s_3\}$ denote the set of site patterns in the
specific case described above of three taxa and two states evolving in a two-state
symmetric model.
Assume that the counting of site pattern $s_i$ is $n_i$. Then
$n=n_0+n_1+n_2+n_3$ is the total length of the sequences and, in the
independent two-state symmetric model considered in this paper,  the
quadruple $(n_0,n_1,n_2,n_3)$ is a sufficient statistics of the
sequence data. We use the letter $\nn$ to denote any quadruple
$(n_0,n_1,n_2,n_3)$ of nonnegative integers such that
$|\nn|=n_0+n_1+n_2+n_3=n\ge1$. 

For every site pattern $s_i$ and every branch lengths $(\tte,\tti)$, 
let $p_i(\tte,\tti)$ denote the
probability that $s_i$ occurs on tree $R_1$ with branch
lengths $(\tte,\tti)$.  Standard computations provided by Yang and Rannala~\cite{yang:branch}
show that
\begin{align*}
4p_0(\tte,\tti) &= 1 + \ee^{-4\tte} + 2 \ee^{-4(\tti+\tte)},\\
4p_1(\tte,\tti) &= 1 + \ee^{-4\tte} - 2 \ee^{-4(\tti+\tte)},\\
4p_2(\tte,\tti) &= 4p_3(\tte,\tti) = 1- \ee^{-4\tte}.
\end{align*}
Let $\gtt=(\gte,\gti)$ denote a pair of positive random
variables representing the branch lengths $(\tte,\tti)$, and
$\NN=(N_0,N_1,N_2,N_3)$ denote a quadruple of integer random variables representing the counts
of sites patterns $\nn=(n_0,n_1,n_2,n_3)$.


\section{The star tree paradox} \label{sect:main}

Assuming that every taxon evolved from a common ancestor, the
aim of phylogeny reconstruction is to compute the most likely tree
$R_i$. To do so, in the Bayesian approach, one places prior
distributions on the trees $R_i$ and on their branch lengths
$\gtt=(\gte,\gti)$. 

\subsection{Main result}
Let $\Prob(\NN=\nn|R_i,\gtt)$ denote the probability that
$\NN=\nn$ assuming that the data is generated along the
tree $R_i$ conditionally on the branch lengths $\gtt=(\gte,\gti)$. One
may consider $R_1$ only since, for 
every $\nn=(n_0,n_1,n_2,n_3)$, the symmetries of the setting yield the relations
$$
\Prob(\NN=\nn|R_2,\gtt) 
= 
\Prob(\NN=(n_0,n_2,n_3,n_1)|R_1,\gtt),
$$
and
$$
\Prob(\NN=\nn|R_3,\gtt) 
= 
\Prob(\NN=(n_0,n_3,n_1,n_2)|R_1,\gtt).
$$
\begin{nota}
For every site pattern $s_i$, let $P_i$ denote the random variable 
$$
P_i=p_i(\gtt)=p_i(\gte,\gti).
$$
For every $i$ in $\tau$ and every $\nn$, let  $\Pi_i(\nn)$
denote the random variable  
$$
\Pi_i(\nn)=P_0^{n_0}P_1^{n_i}P_2^{n_j+n_k},
\quad\mbox{with}\quad\{i,j,k\}=\tau.
$$
\end{nota}
We recall that $P_2=P_3$ and we note that, if 
$|\nn|=n_0+n_1+n_2+n_3=n$ with $n\ge1$, then, for every $i$ in $\tau$,
$$
\Pi_i(\nn)=P_0^{n_0}P_1^{n_i}P_2^{n-n_0-n_i}.
$$
Fix $\nn$ and assume that $|\nn|=n_0+n_1+n_2+n_3=n$ with $n\ge1$. 
For every $i$ in $\tau$, the posterior probability of $R_i$ conditionally
on $\NN=\nn$ is 
$$
  \Prob(R_i|\NN=\nn) 
= \frac{n!}{n_0!n_1!n_2!n_3!}  \,\frac1{\Prob(\NN=\nn)}\,\E(\Pi_i(\nn)).
$$
Thus, for every $i$ and $j$ in $\tau$,
$$
  \frac{\Prob(R_i|\NN=\nn)}{\Prob(R_j|\NN=\nn)} 
= 
\frac{\E(\Pi_i(\nn))}{\E(\Pi_j(\nn))}.
$$

For every $\varepsilon>0$ and every $i$ in $\tau$, let
$\caln_i^\varepsilon$ denote the set of $\nn$ such that, for both indices 
$j$ in $\tau$ such that $j\ne i$,
$$
\E(\Pi_i(\nn))\geq (2/\varepsilon)\,\E(\Pi_j(\nn)).
$$
One sees that, for every $i$ in $\tau$ and $ \nn$ in $\caln_i^\varepsilon$,
$$
  \Prob(R_i|\NN=\nn) \geq 1 - \varepsilon,
$$
which means that the posterior probability of tree $R_i$ among the
three possible trees is highly supported.

Recall that, under hypothesis $R_0$ and for a tame prior distribution
on $\gtt=(\gte,\gti)$, Steel and Matsen prove that, 
for every $i$ in $\tau$, $\Prob(\NN\in\caln_i^\varepsilon)$ does not go
to $0$ when the sequence length $n$ 
goes to infinity, and consequently that the posterior probability
$\Prob(R_i|\NN)$ can be close to $1$ even when the sequence length $n$ is
large.

As stated in the introduction, our aim is to prove the same result
for tempered prior distributions of $\gtt=(\gte,\gti)$, which we now define.


\begin{nota}
(1) For every $s\in[0,1]$ and $z\in[0,3]$, let
$$
 G(z,s) =  \Prob \left( \ee^{-4 \gte}(1-\ee^{-4\gti}) \le s \,|\,
   \ee^{-4 \gte}(1+2\ee^{-4\gti}) = z \right). 
$$
(2) For every positive $t$ and every site pattern $s_i$, let  $q_i$ denote
the probability that $s_i$ occurs on tree $R_0$, hence 
$$
4q_0=4p_0(0,t)=1+3\ee^{-4t},\quad
4q_1=4q_2=4q_3=1-\ee^{-4t}.
$$
(3) Let $\ell_t$ denote a positive real number such that $1< 4q_0 -\ell_t$
and $4q_0+\ell_t< 4$, for instance $\ell_t = 3\ee^{-4 t}\,(1 -\ee^{-4
  t}) $. Let $I$ and $I_t$ denote the intervals 
$$
I = [ 0, 3 ], \quad
I_t = [4q_0-1 - \ell_t , 4q_0-1 +
\ell_t] \subset]0,3[.
$$
(4)
For every positive $t$ and integer $n$, let
$$
Q_n(t)=\Prob \left( \gti\le1/n,\,t\le\gte\le
    t+1/n \right).
$$
\end{nota}

\begin{defi}[Tempered priors] \label{defi:temp} The distribution of
  $\gtt=(\gte,\gti)$ is tempered if the following two conditions hold.
\begin{enumerate}
\item\label{deftemp1}
For every $t$, there exists a real number $s_0$ in $]0,1]$, an interval $I_t$ around
$4q_0-1$, some bounded functions 
$F_i$, some positive numbers $\alpha$ and $\kappa$, an integer $k\ge1$ and some real numbers
$\varepsilon_i$ such that
\[
0=\varepsilon_0 < \varepsilon_1 < \dots < \varepsilon_{k-1} \le 2 <
\varepsilon_{k},
\]
and such that for every $s$ in $[0,s_0] $ and every $ z$ in $I_t$,
$$
\left| G(z,s) - \sum_{i=0}^{k-1} F_i(z)s^{\alpha+\varepsilon_i}
\right| \le \kappa s^{\alpha + \varepsilon_k}.
$$
\item\label{deftemp2} 
For every positive $t$, $n^{-1}\log Q_n(t)\to0$ when $n\to\infty$.
\end{enumerate}
\end{defi}
We detail the properties involved in Definition~\ref{defi:temp} and
provide examples of tempered priors in subsection~\ref{subsect:temp} below.


We now state our main result, which is an extension of Steel and
Matsen's result to our more general setting.

\begin{theo} \label{theo:baypar}
  Consider sequences of length $n$ generated by a star tree $R_0$ on
  $3$ taxa with strictly positive edge length $t$. Let $\NN$ be the
  resulting data, summarized by site pattern counts. Consider any
  prior on the three resolved trees $(R_1, R_2, R_3)$ which assigns
  strictly positive probability to each tree, and a tempered
  prior distribution on their branch lengths $\gtt=(\gte,\gti)$.
\\
Then, for every $i$ in $\tau$ and every positive
  $\varepsilon$, there exists a positive $\delta$ such that,
  when $n$ is large enough,
\[
  \Prob \left( \Prob(R_i|\NN) \ge1-\varepsilon \right) \ge \delta.
\]
\end{theo}

We prove Theorem~\ref{theo:baypar} in Section~\ref{sect:theoproof}. 

\subsection{Motivation and intuitive understanding of
  Definition~\ref{defi:temp}} \label{subsect:temp}

In Definition~\ref{defi:temp}, condition~\ref{deftemp2}
is easy to describe,
to illustrate and to check, while the content of condition~\ref{deftemp1} might be more difficult to
grasp. Condition~\ref{deftemp1} involves a Taylor expansion
around $s=0$ of the conditional cumulative distribution function $s \mapsto G(z,s)$, where the Taylor coefficients 
depend on $z$. Such a Taylor expansion roughly describes the prior distribution when $\tti\to0$ and when $\tte$ is roughly constant. The precise definition of $G(z,\cdot)$ and the technical result
stated in Proposition~\ref{prop:main} are both dictated by our approach to the proof of Theorem~\ref{theo:baypar}. A key hypothesis is that $\varepsilon_0=0$ while $\varepsilon_k>2$, which means that we are given a limited expansion of $s\mapsto G(z,s)$ up to a better order than $s^2$ when $s\to0$.

At this point, the reader can wonder how to check if a given prior is tempered or
not and if the verification is simply possible in concrete cases, given the convoluted 
aspect of this definition. Hence we now present some explicit examples of
tempered priors. We begin with the following result.

\begin{prop} \label{prop:tame} 
Assume that $\gtt=(\gte,\gti)$ has a smooth
  joint probability density, bounded and everywhere
  non zero. Then the distribution of $\gtt=(\gte,\gti)$ is
  tempered.
\end{prop}

As a consequence, every tame prior fulfills the hypothesis of
Proposition~\ref{prop:tame}, hence every tame prior is tempered, as
claimed in the introduction. This case includes the exponential 
priors discussed in~\cite{yang:branch}. We 
prove Proposition~\ref{prop:tame} in Appendix~\ref{appe:tame}. 

However some
tempered priors are not  
tame, as illustrated by the following example where 
Steel and Matsen's condition fails.

\begin{defi}\label{defab}
Let $a>0$ and $b>0$. Let $(t_n)$, $(y_n)$ and $(r_n)$ denote sequences
of positive numbers, indexed by $n\ge1$ and
defined by the formulas
$$
  t_n = n^{-a}, \quad 
y_n = 1+2\ee^{-4t_n} \quad 
r_n = y_n\, \left( n^{-b} - (n+1)^{-b}\right).
$$
Finally, let
$$
r= \sum_{n \geq 1} r_n.
$$
\end{defi}
\begin{prop} \label{prop:disc}
In the setting of Definition~\ref{defab}, assume the following:
\begin{itemize}
 \item[(i)]
$3a<\min\{1,b\}$.
\item[(ii)] The random variable $\gti$ is discrete and such that, for every $n\ge1$,
\[
  \Prob( \gti = t_n ) =  r_n/r. 
\]
\item[(iii)] The random variable $\gte$ is continuous, independent of
$\gti$, with exponential law of parameter $4$, that is, with density
$4\,\ee^{-4t}$ on $t\ge0$ with respect to the Lebesgue measure.
\end{itemize}
Then, the distribution of  $\gtt=(\gte,\gti)$ is not tame but it is
tempered, for the parameters 
\[
  k=3, \quad \alpha = b/a, \quad \varepsilon_1 = 1, \quad
  \varepsilon_2 = 2, \quad \varepsilon_3 = 3.
\]
\end{prop}

Since the distribution of $\gti$ is an accumulation of Dirac
masses, the prior distribution of 
$\gtt=(\gte,\gti)$ cannot be tame.

Yet, the fact that the prior distribution is tempered does not come
only from the fact that the distribution of $\gti$ 
is discrete. For a degenerate example, if $\gti=0$ almost surely, 
then $G(z,s)=1$  for every $s\ge 0$, and $G(z,\cdot)$ has no
Taylor expansion around zero whose first 
term is a positive
power of $s$. Note that in this particular case, the
Bayesian star paradox does not occur. 

However,  under the conditions of 
Proposition~\ref{prop:disc}, $G(z,\cdot)$ has a Taylor
expansion at $0$ fulfilling condition~\ref{deftemp1} of
Definition~\ref{defi:temp}. We prove this in
Appendix~\ref{appe:tame}. 

We provide below some examples of less ill-behaved distributions which
are tempered but not tame, and an example of a distribution which does
not fulfill condition~\ref{deftemp1}, hence is not tempered.

\begin{prop} \label{prop:examtemp}
Assume that $\gte$ is a continuous random variable,  with exponential
law of parameter $4$, that is, with density 
$4\,\ee^{-4t}$ on $t\ge0$ with respect to the Lebesgue measure, and that $\gti$ is a random variable independent of $\gte$. Then, the following holds.
\begin{itemize}
 \item[(i)] If the distribution of $\gti$ is uniform on $[0,\theta]$, with $\theta
>0$, the distribution of $\gtt=(\gte,\gti)$ is tempered but not tame.
\item[(ii)] 
If the distribution of $\gti$ has density $\theta\tti^{\theta-1}$
on the interval $[0,1]$, for a given $\theta$ in $(0,1)$, 
the distribution of
$\gtt=(\gte,\gti)$ is tempered but not tame.
\item[(iii)] 
If the distribution of $\gti$ has density $\log (1/\tti)$ on the interval $[0,1]$, the
distribution of $\gtt=(\gte,\gti)$ 
is not tempered.
\item[(iv)] 
If the distribution of $\gti$ has density $4 \, \tti \log (1/\tti)$ on the interval $[0,1]$, the
distribution of $\gtt=(\gte,\gti)$ 
is not tempered.
\end{itemize}
\end{prop}

Note that in case $(iv)$, the density function of $\gtt=(\gte,\gti)$
is bounded, non
smooth but continuous, but the distribution is not tempered.

We prove Proposition~\ref{prop:examtemp} in Appendix~\ref{appe:tame}.

\section{Extension of Steel and Matsen's lemma}
\label{sect:sm.etendu}
The Bayesian star paradox due to Steel and Matsen relies on a
technical result which we slightly rephrase as follows. For every
nonnegative real $t$ and every $[0,1]$ valued random 
variable $V$, introduce
\[
M_t=\E(V^t),\quad  
R_t= 1-\frac{M_{t+1}}{M_t}=\frac{\E \left( V^t(1-V)\right)}{\E(V^t)}.
\]
\begin{prop}[Steel and Matsen's lemma]
  Let $0\leq \eta < 1$ and $B>0$. There exists a finite $K$, which
  depends on $\eta$ and $B$ only, such that the following holds. For
  every $[0,1]$ valued random variable $V$ with a smooth probability
  density function $f$ such that $f(1) > 0$ and $|f'(v)| \leq B f(1) $
  for every $ \eta \leq v \leq 1$, and for every integer $ k \geq K$,
\[
2 k R_k \geq 1.
\]
\end{prop}
Indeed the asymptotics of $R_k$ when $k$ is large depends on the behaviour of the
distribution of $V$ around $1$. 

Our next proposition proves
that the conclusion of Steel and Matsen's lemma above holds for a
wider class of random variables. 

\begin{prop} \label{prop:main}
Let $V$ a random variable on  $ [0,1] $. Suppose that there exists an integer $n\ge1$ and
real numbers $0\le v_0 <1 $, $ \alpha>0 $, 
$\varepsilon_i$ and $\gamma_i$, such that
\[
  0=\varepsilon_0 < \varepsilon_1 < \dots < \varepsilon_{n-1} \leq 1 <
  \varepsilon_{n}, 
\]
and, for every $v_0\le v\le1$,
$$
  \left| \Prob ( V \geq v ) - \sum_{i=0}^{n-1} \gamma_i ( 1 - v
    )^{\alpha+\varepsilon_i} \right| 
  \leq
  \gamma_n (1-v)^{\alpha + \varepsilon_n }.
$$
Then there exists a finite $\theta(\gamma)$, which depends
continuously on $\gamma=(\gamma_0,\ldots,\gamma_n)$, such that for every $ t \ge
\theta(\gamma) $, 
$$
2t R_t\geq\alpha.
$$
\end{prop}

\begin{rema}
  We insist on the fact that $\theta(\gamma)$ depends continuously
  on the multiparameter $\gamma=(\gamma_0,\ldots,\gamma_n)$. To wit, in the proof of Proposition~\ref{prop:Z},
  we apply Proposition~\ref{prop:main} with bounded functions of $z$. This
  means that for every $z$ in $I_t$, one gets a number $\theta$ which
  depends on $z$ through the bounded functions such that the control on the
  distribution of $V$  
holds. The continuity of $\theta$ ensures
  that there exists a number independent of $z$ such that
  Proposition~\ref{prop:Z} holds.
\end{rema}

\begin{rema}
If one computes a Taylor expansion of the function $v \mapsto \Prob (
V \geq v )$ at $v=1^-$ under the conditions of Steel and Matsen's
lemma, one sees that conditions of Proposition~\ref{prop:main}
hold. Hence Proposition~\ref{prop:main} is an extension of
Steel and Matsen's lemma.  
\end{rema}

We prove Proposition~\ref{prop:main} in Appendix~\ref{appe:main}. The
proof of Theorem~\ref{theo:baypar} relies on it. 

\section{Synopsis of the proof of
  Theorem~\ref{theo:baypar}} \label{sect:theoproof} 

This section is devoted to a sketch of the proof of
Theorem~\ref{theo:baypar}. We use the definitions below. 
Note that the set $F_c^{(n)}$ defined below is not the set introduced by Steel and
Matsen. For a 
technical reason in the proof of Proposition~\ref{prop:claim1} stated
below, we had to modify 
their definition. Note however that Propositions~\ref{prop:claim1} and
\ref{prop:claim2} below are adaptations of ideas in Steel and Matsen's
paper. 

\begin{nota}
Define functions $\Delta_i$ as follows. For every nonnegative integers
  $\nn=(n_0,n_1,n_2,n_3)$ such that
  $|\nn|=n_0+n_1+n_2+n_3=n$ with $n\ge1$,
$$
  \Delta_0(\nn) = \frac{n_0 - q_0n}{\sqrt{n}},
$$
and, for every $i$ in $\tau$,
$$
  \Delta_i(\nn) = \frac{n_i - 1/3(n-n_0)}{\sqrt{n}}.
$$
For every $c>1$, introduce
\[
  F_c^{(n)} = \{\nn\,;\,|\nn|=n,\,-2c \leq
  \Delta_2(\nn) \leq -c,\,-2c \leq \Delta_3(\nn) \leq
  -c,\,-c \leq \Delta_0(\nn) \leq 0\}. 
\]
For every $i$ in $\tau$ and every positive $\eta$, let $A^i_\eta$ denote the event
\[
  A^i_\eta = \left\{ \forall j \in \tau,\,j\ne i, \,\E(\Pi_i(N)\,|\,N) \geq
    \eta \E(\Pi_j(N)\,|\,N) \right\}   . 
\]
\end{nota}

Since $\Delta_1 + \Delta_2 + \Delta_3 = 0$, every $\nn$ in $
F_c^{(n)} $ is such that $2c \leq \Delta_1(\nn) \leq 4c$. We note
that $F_c^{(n)}$ is not symmetric about $\tau$ and gives a preference
to $1$. That is why we only deal with $A^1_\eta$ in the following
proof. To deal with $A^i_\eta$, one would change the definition $F_c^{(n)}$ accordingly. 

From the reasoning in Section~\ref{sect:main}, it suffices to prove
that for every positive $\eta$, there exists a positive $\delta$ such
that, when $n$ is large enough,  
$$
  \Prob \left( A^1_\eta \right) \geq \delta.
$$ 

  Suppose that one generates $n\ge1$ sites on 
the star tree $R_0$ with given branch length $t$ and let $\NN=(N_0,N_1,N_2,N_3)$ denote the
  counts of site patterns defined in
  Section~\ref{sect:bayfra}, hence $N_0+N_1+N_2+N_3=n$.

From central limit estimates, 
the probability of the event $\left\{\NN \in F_c^{(n)}\right\}$ is
uniformly bounded from below, say by $\delta>0$, when $n$
is large enough.
Hence,
\[
  \Prob \left( A^1_\eta  \right)
\geq \delta
  \Prob \left( A^1_\eta \big| \, \NN \in F_c^{(n)} \right)
\]
We wish to prove that there exists a positive $\alpha$ independent of
$c$ such that for $n$ large enough  
and for every $\nn$ in $F_c^{(n)}$ and for $j=2$ and $j=3$,
\[
  \E(\Pi_1(\nn))\geq c^2 \alpha\,\E(\Pi_j(\nn)).
\]
This follows from the two results below, adapted from Steel and
Matsen's paper. 
\begin{prop}\label{prop:claim1}
Fix $t$ and assume that $\nn$ is in $F_c^{(n)}$. Then, when $n$ is
large enough, for $j=2$ and $j=3$,  
$$
\E( \Pi_j(\nn) \,|\, 4P_0-1 \in I_t  ) \geq  \E( \Pi_j(\nn)
\,|\, 4P_0-1 \notin I_t ). 
$$
\end{prop}
\begin{prop}\label{prop:claim2}
Fix $t$ and assume that $\nn$ is in $F_c^{(n)}$. Then,
  there exists a positive $\alpha$, independent of $c$, such that for
  every $z$ in $I_t$, and for $j=2$ and $j=3$,
$$
\E( \Pi_1(\nn) \,|\, 4P_0-1 = z  ) \geq c^2 \alpha \E(
\Pi_j(\nn) \,|\, 4P_0-1 = z  ). 
$$
\end{prop}
We prove Propositions~\ref{prop:claim1} and~\ref{prop:claim2} in
Section~\ref{sect:claim}. 
  
From these two results, for $j=2$ and $j=3$,
\[
  \E(\Pi_1(\nn)) \geq c^2 \alpha \Prob( 4P_0-1 \in I_t
  )\,\E(\Pi_j(\nn)). 
\]
Assume that $c$ is so large that $c^2\alpha \Prob( 4P_0-1 \in I_t
)\ge\eta$. Then,  
 for every $\nn$ in $F_c^{(n)}$ and for $j=2$ and $j=3$,
\[
  \E(\Pi_1(\nn)) 
\geq 
\eta\,\E(\Pi_j(\nn)).
\]
This implies that
$\Prob \left( A^1_\eta \big| \, \NN \in F_c^{(n)} \right) = 1,$
which yields the theorem.


\section{Proofs of Propositions~\ref{prop:claim1}
  and~\ref{prop:claim2}} \label{sect:claim}
 
\subsection{Proof of Proposition \ref{prop:claim1}} \label{sect:claim1}

The proof is decomposed into two intermediate results, stated as
lemmata below and using estimates on auxiliary random variables
introduced below. 

\begin{nota}
For every $n\ge1$ and $t>0$, let $\Gamma_t(n) = [0,1/n] \times [t,t+1/n]$.
\\
For every $t>0$, let $\mu_t =
q_0^{q_0}q_1^{q_1}q_2^{q_2}q_3^{q_3}=q_0^{q_0}q_1^{3q_1}$ 
and $U_t$ denote the random variable
$$
U_t=\prod_{i=0}^3(P_i/q_i)^{q_i}.
$$
For every $\nn$ and for $ j=2$ and $j=3$, let $W_j(\nn)$
denote the random variable 
$$
W_j(\nn)=P_0^{\Delta_0(\nn)} P_1^{(\Delta_j - \Delta_0 /3
  )(\nn) } P_2^{ (\Delta_1 + \Delta_k - 2\Delta_0 /3)(\nn)
}, \quad \mbox{with} \quad \{j,k\}=\{2,3\}. 
$$
\end{nota}
One sees that
$$
U_t=P_0^{q_0}P_1^{q_1}P_2^{2q_1}/\mu_t,\qquad
Q_n(t)=\Prob
\left( \gtt \in \Gamma_t(n) \right),
$$
and, for $j=2$ and $j=3$,
$$
W_j
=(P_0/P_2)^{\Delta_0}(P_1/P_2)^{\Delta_j - \Delta_0 /3}.
$$

\begin{lemm}\label{lemm:estim}
(1) For every $\nn$ in $F_c^{(n)}$ and for $ j=2$ and $j=3$,
$W_j(\nn)\le1$. 
\\
(2) For every $\nn$ in $F_c^{(n)}$ and for $ j=2$ and $j=3$,
$W_j(\nn)\ge (q_1)^c$ on the event $\{\gtt\in\Gamma_t(n)\}$. 
\\
(3)
There exists a finite constant $\kappa$ such that
$U_t^n\ge\ee^{-\kappa}$ uniformly on the integer $n\ge1$ and on the event
$\{\gtt\in\Gamma_t(n)\}$. 
\end{lemm}

\begin{proof}[Proof of Lemma~\ref{lemm:estim}]
(1) For every $\gtt$, $P_0\ge P_1\ge P_2$. On $F_c^{(n)}$,
$\Delta_0\le0$ and for $ j=2$ and $j=3$, $\Delta_j - \Delta_0
/3\le0$ hence  
$$
(P_0/P_1)^{\Delta_0}\le1,\quad
(P_0/P_2)^{\Delta_j - \Delta_0 /3}\le1.
$$
This proves the claim.

(2)
One has $P_0\le1$ everywhere and $P_1\ge q_1$ and $P_2\ge q_1$ on the
event $\{\gtt\in\Gamma_t(n)\}$. 
On $F_c^{(n)}$, $\Delta_0\le0$ and for $ j=2$ and $j=3$,
$\Delta_j - \Delta_0 /3\le0$ hence  
$W_j\ge q_2^{-\Delta_j - 2\Delta_0 /3}$. Finally, on $F_c^{(n)}$,
$\Delta_j + 2\Delta_0 /3\le-c$. This proves the claim. 

(3)
For every $\gtt$ in $\Gamma_t(n)$, one has $\gti\ge0$ and $\gte\ge t$, hence
$P_1\ge q_1$ and $P_2\ge q_2=q_1$. Likewise, $\gti\le1/n$ and $\gte\le
t+1/n$ hence 
$$
P_0\ge p_0(1/n,t+1/n)\ge
q_0-5\ee^{-4t}(1-\ee^{-4/n})/4.
$$ 
This yields that, for every $n\ge1$
and every $\gtt$ in $\Gamma_t(n)$, 
$$
U_t^n\ge(1-5\ee^{-4t}/(q_0n))^n\to\exp(-5\ee^{-4t}/q_0)>0,
$$
which implies the desired lower bound.
\end{proof}

\begin{lemm}\label{lemm:lemm} 
For every $\nn$ in $F_c^{(n)}$ and for $ j=2$ and $j=3$,
$$
\E(\Pi_j(\nn) \,|\, 4P_0-1 \in I_t ) \geq \mu_t^n Q_n(t) \exp(-O(\sqrt{n})),
$$
and
$$
\E(\Pi_j(\nn) \,|\, 4P_0-1 \notin I_t ) \leq \mu_t^n \exp(-n\ell_t^2/32).
$$
\end{lemm}
\begin{proof}[Proof of Lemma~\ref{lemm:lemm}]
Since $P_0=p_0(\gtt)$, for every $\gtt$ in $\Gamma_t(n)$, when $n$ is
large, $4P_0-1$ is in the interval $I_t$. Consequently, 
\[
  \E(\Pi_j(\nn) \,|\, 4P_1-1 \in I_t ) \geq Q_n(t) \E \left(
    \Pi_j(\nn) \,|\, \gtt \in \Gamma_t(n)  \right). 
\]
On the event
$\{\gtt \in \Gamma_t(n)\}$,
$$ 
  \Pi_j(\nn) = \mu_t^nU_t^n
  W_j(\nn)^{\sqrt{n}}\ge\mu_t^n\ee^{-\kappa}(q_1)^{c\sqrt{n}}, 
$$
from parts (2) and (3) of Lemma~\ref{lemm:estim}, which proves the
first part of the lemma. 

Turning to the second part, 
let $\dd_{KL}$ denote the Kullback-Leibler distance between
discrete probability measures. 
When $ 4P_0-1$ is not in $I_t$,
\[
\dd_{KL} (q,P) \geq (1/2) \|q - P\|_{1}^2
\geq (1/2)(q_0 - P_0)^2 \geq \ell_t^2/32.
\]
Note that
$$
  \Pi_j(\nn) = \mu_t^n W_j(\nn)^{\sqrt{n}} \exp( - n
    \dd_{KL} (q,P)), 
$$
hence the estimate on $\dd_{KL} (q,P)$, and part (1) of
Lemma~\ref{lemm:estim}, imply the second part of the lemma. 
\end{proof}

Turning finally to the proof of Proposition \ref{prop:claim1}, we note
that $Q_n(t)=\ee^{o(n)}$ because the distribution of $\gtt$ is
tempered. 
Furthermore, Lemma~\ref{lemm:lemm} shows that, when $n$ is large enough,
\[
  \E( \Pi_j(\nn) \,|\, 4P_0-1 \in I_t ) \geq  \E( \Pi_j(\nn)
  \,|\, 4P_0-1 \notin I_t),
\]
and this concludes the proof of Proposition \ref{prop:claim1}.


\subsection{Proof of Proposition \ref{prop:claim2}} \label{sect:claim2}

Our proof of Proposition \ref{prop:claim2} is based on
Lemma~\ref{lemm:A/B} and Proposition~\ref{prop:Z} below. 

\begin{nota}
For every $u$  in $[0,1]$, let $\zeta(u) = (1 + 2u) (1 - u)^2 $.
Let $U$ and $V$ denote the random variables defined as
\[
  U = (P_1 - P_2)/(1 - P_0),
  \qquad 
  V=\zeta(U).
\]
\end{nota}
\begin{lemm} \label{lemm:A/B}
For every $\nn$ in $F_c^{(n)}$ and for $ j=2$ and $j=3$,
\[
  \frac{\E(\Pi_1(\nn)\,|\,P_0)}{\E(\Pi_j(\nn)\,|\,P_0)} 
  \geq 
  4c^2 n 
  \frac{\E \left( V^s (1-V) \,|\,P_0\right)}{\E \left( V^s\,|\,P_0
    \right)},\quad \mbox{where}\
s=(n-n_0)/3.
\]
\end{lemm}
\begin{proof}[Proof of Lemma~\ref{lemm:A/B}]
Recall that, for every $c>1$, $F_c^{(n)}$ is
\[
  F_c^{(n)} = \{\nn \,:\,|\nn|=n,\,-2c \leq
  \Delta_2(\nn) \leq c,-2c \leq \Delta_3(\nn) \leq c,-c \leq
  \Delta_0(\nn) \leq 0\}. 
\]
Using the $\Delta$ variables, one can rewrite $\Pi_1$, $\Pi_2$ and $\Pi_3$ as
$$
\Pi_i(\nn)
= P_0^{n_0} (P_1 P_2^2)^{s} \left( P_1/P_2 \right)^{\Delta_i(\nn)
  \sqrt{n}},\quad 
i=1,2,3,\ s=(n-n_0)/3.
$$
Assume that $\nn$ is in $F_c^{(n)} $. Then, $\Delta_1(\nn)\ge2c$,  $\Delta_j(\nn)\le0$ for $j=2$ and $j=3$, and $P_1\ge P_2$. Hence 
$$
\Pi_1(\nn)
\geq P_0^{n_0} (P_1 P_2^2)^{s} \left(P_1/P_2\right)^{2c \sqrt{n}},
\qquad
\Pi_j(\nn) 
\leq P_0^{n_0} (P_1 P_2^2)^{s}.
$$
Furthermore,
$$
 P_1 P_2^2 = (1/27) V (1 - P_0)^3,
  \quad 
  P_1/P_2 = (1+2U)/(1-U),
$$
  hence for $j=2$ and $j=3$,
$$
  \frac{\E(\Pi_1(\nn)\,|\,P_0)}{\E(\Pi_j(\nn)\,|\,P_0)} \geq
  \frac{\E \left( V^{s} \left( (1+2U)/(1-U)\right)^{2c\sqrt{n}}\,
      |\,P_0\right) }{\E \left( V^{s} \,|\,P_0\right)}. 
$$
Direct computations (or Lemma 3.2 in Steel and Matsen \cite{steel:baypar}) 
show that,
 for every $u$ in $[0,1)$ and every $m \geq 3$,
\[
  \left( (1+2u)/(1-u)  \right)^m \geq m^2 ( 1 - \zeta(u)),
\]
hence
$$
\left( (1+2U)/(1-U)\right)^{2c\sqrt{n}}\ge4c^2n\,(1-V).
$$
The conclusion of Lemma \ref{lemm:A/B} follows.
\end{proof}

\begin{prop} \label{prop:Z} Assume that the
  distribution of $\gtt$ is tempered. Then there exists $\theta$ and
  $\alpha$, both positive and independent of $c$, such that for every $s
  \ge \theta$, on the event $\{4P_0-1\in I_t\}$,
\[
  4s \,\E(V^s(1-V)\,|\,P_0) \geq \alpha\,\E(V^s\,|\,P_0).
\]
\end{prop}
\begin{proof}[Proof of Proposition~\ref{prop:Z}]
We recall that $U$ and $V$ denote random variables defined as
\[
  U = (P_1 - P_2)/(1 - P_0),
  \quad V = \zeta(U),\quad
  \zeta(u)=(1 + 2u) (1 - u)^2.
\]
To use Proposition~\ref{prop:main}, one must compute a Taylor
expansion at $v=1^-$ or, equivalently, at $u=0^+$, of the conditional
probability 
\[
  \Prob (V \geq v \,|\, P_0) = \Prob( U \leq u \,|\, P_0),
\]
where $u=\zeta^{-1}(v)$. 
Besides, for $v$ close to $1$,
$$u=\zeta^{-1}(v) = w/\sqrt{3}+
  w^2/9 + 5w^3/54\sqrt{3} + O(w^4),\quad \mbox{with } w=\sqrt{1-v}.
$$
Since $ U = (P_1 - P_2)/(1 - P_0) $,
\[
  \Prob( U \leq u \,|\, 4P_0-1 = z) = \Prob \left( \sse(3 - \ssi) \leq
    2s \,|\, \sse\ssi = z\right), 
\]
where we used the notations
\[
  \sse=\ee^{-4\gte}, \quad \ssi=1 + 2\ee^{-4\gti}, \quad  2s=u(3-z).
\]
Using Definition~\ref{defi:temp}, one has
\[
  G(z,s)= \Prob \left( \sse(3 - \ssi) \leq 2s \,|\, \sse\ssi=z \right).
\]
Since the distribution of $\gtt$ is tempered, there exists some
bounded functions $F_i$ defined on $I_t$, a positive number
$\alpha$, $n+1$ real numbers
\[
  0=\varepsilon_0 < \varepsilon_1 < \dots < \varepsilon_{n-1} \leq 2 <
  \varepsilon_{n}, 
\]
and two positive numbers $\kappa$ and $s_0$ such that for every $ 0
\leq s \leq s_0 $ and every $z$ in $I_t$,  
$$  \left| G(z,s) - \sum_{i=0}^{n-1} F_i(z)s^{\alpha+\varepsilon_i}
\right| \leq \kappa s^{\alpha + \varepsilon_n}. 
$$
Combining this with the relation $2s=u(3-z)$ and the expansion of
$u=\zeta^{-1}(v)$ along the powers of $w$, 
one sees that there exists
some bounded functions $f_i$ on $I_t$, a positive number
$\kappa'$ and $0\le v_0<1$ such that for every  $v_0\le v\le1$ and
every $ z\in I_t$,  
$$
  \left| \Prob (V \geq v \,|\, 4P_0-1 = z)  - \sum_{i=0}^{n-1}
    f_i(z)(1-v)^{\alpha/2+\varepsilon_i/2} \right| \leq \kappa'
  (1-v)^{\alpha/2 + \varepsilon_n/2}. 
$$
Since the functions $f_i$ are bounded and positive on $I_t$,
Proposition~\ref{prop:main} implies
that there exists a positive number 
$\theta$ such that for every $z$ in $I_t$ and every $ s \ge \theta$,
the conclusion of Proposition~\ref{prop:Z} holds.
\end{proof}

Assuming this, the proof of Proposition~\ref{prop:claim2} is as follows.
Let $s$, $\theta$ and $\alpha$ be as in Lemma~\ref{lemm:A/B} and
Proposition~\ref{prop:Z}. 
Since $n-n_0 = (1-q_0)n - \Delta_0 \sqrt{n} \geq (1-q_0)n$ for every
$\nn$ in $F_c^{(n)}$, one knows that
 $s=(n-n_0)/3 \ge \theta$ when $n$ is large enough. Furthermore, $s\le n/3$.
Finally, for every $\nn$ in $F_c^{(n)}$ with $n$ large enough, on
the event $\{4P_0-1\in I_t\}$ and for $ j=2$ and $j=3$, 
\[
  \E(\Pi_1(\nn)\,|\,P_0) \geq 3c^2 \alpha\,\E(\Pi_j(\nn)\,|\,P_0).
\]
This concludes the proof of Proposition \ref{prop:claim2}.


\section*{Acknowledgments}

I would like to thank Mike Steel and an anonymous referee for some
helpful comments. 


\bibliographystyle{plain}
\bibliography{references4}

\appendix

\section{Proof of Propositions~\ref{prop:tame},~\ref{prop:disc},
  and~\ref{prop:examtemp}} \label{appe:tame}

\begin{nota}
Introduce the random variables
$$
(\sse,\ssi)=\varsigma(\gte,\gti),\qquad\mbox{where}\ \varsigma(\tte,\tti)=(\ee^{-4
  \tte},1+2 \ee^{-4 \tti}), 
$$
that is,
$$
\sse=\ee^{-4 \gte}, \qquad \ssi= 1+2 \ee^{-4 \gti}.
$$
\end{nota}
Hence, $G(z,\cdot)$ is defined by
$$
G(z,s)=\Prob \left( 3\sse\le 2s+z\,|\, \sse\ssi = z \right),
$$
\subsection{Proof of Proposition~\ref{prop:tame}}
The distribution of $(\sse,\ssi)$ has a smooth joint probability
density, say $\varpi$, 
  defined on the set $0<x\le1<y\le3$ by
\[
  \varpi(x,y) = \frac{\omega \circ \varsigma^{-1} (x,y)}{16 x (y -1)}.
\]

For tame priors, the probability $Q_n(t)$ introduced in condition~\ref{deftemp2}
of Definition~\ref{defi:temp} is of order $1/n^2$, hence this condition holds.

The definition of $G(z,s)$ as a conditional expectation can be rewritten as
$$
G(z,s) =  \Prob \left( 3\sse\le 2s+\sse\ssi \,|\, \sse\ssi= z \right). 
$$
Hence, for every measurable bounded function $H$,
\[
  \E \left( H(\sse\ssi)\,;\,3\sse \le 2s+\sse\ssi \right)  =\E \left(
    H(\sse\ssi) G(\sse\ssi,s) \right), 
\]
that is,
\[
  \iint H(xy) \1\{ 3x \leq 2s+xy \} \varpi(x,y) \dd x \dd y 
= \iint H(xy) G(xy,s) \varpi(x,y) \dd x \dd y.
\]
The change of variable $z=xy$ yields
$$
  \iint H(z) \1\{ 3x\leq 2s+z \} \varpi \left( x,z/x \right) \dd z \dd x  / x
=
\iint H(z) G(z,s) \varpi \left(x,z/x \right) \dd z \dd x  / x.
$$
This must hold for every measurable bounded function $H$, hence one can choose
\[
  G(z,s) = H(z,s) /  H(z,\infty),
\]
with
\[
H(z,s) 
  = 
  \int \1\{ 3x \leq 2s + z\} \varpi \left( x,z/x \right) \dd x /x.
\]
Since $0\leq \sse \leq 1 \leq \ssi \leq 3$ almost surely, the integral
defining $H(z,s)$ may 
be further restricted to the range $0\le x\leq 1$ and $z/3 \leq
x \leq z $.
Finally, for every $s \ge0$ and $z\in [0,3]$,
\[
  G(z,s) = H(z,s) /  H(z,1),
\]
where
\[
  H(z,s) = \int_{m(0,z)}^{m(s,z)} \varpi \left( x, z/x \right) \dd x /x,
  \quad \mbox{with } m(s,z) = \min\{1,z, (2s+z)/3\}.
\]
Hence,
$m(0,z)=z/3$ and, for small positive values of $s$, $m(s,z)=m(0,z)+2s/3$.
When $0\le z\le1$, $m(s,z)\to m(\infty,z)=z$ when $s\to\infty$ and
this limit is reached for $s=z$. 
When $1\le z\le3$, $m(s,z)\to m(\infty,z)=1$ when $s\to\infty$ and
this limit is reached for $s=(3-z)/2$. 
In both cases, $m(\infty,z)=m(1,z)$ hence $H(z,\infty)=H(z,1)$.

Because $\omega$ and $\varsigma^{-1}$ are smooth,  the Taylor-Lagrange
formula shows that,  
for every $s\ge0$ and every fixed $z$,
\[
  H(z,s) = 
  H(z,0) + H'(z,0) s + H''(z,0) \frac{s^2}2 + H^{(3)}(z,0) \frac{s^3}6 +
  \int_0^s (x-s)^3 H^{(4)}(z,s) \frac{\dd x}{24},
\] 
where all the derivatives are partial derivatives 
with respect to the second argument $s$.

Simple computations yield
$H(z,0) = 0$ and the values of the three derivatives $H'(z,0)$, $H''(z,0)$ and
$H^{(3)}(z,0)$ as combinations of $\omega$ and of partial derivatives
of $\omega$, evaluated at the point $(\vartheta,0)$, where
$3\ee^{-4\vartheta}=z$.

Furthermore, the hypothesis on $\omega$ ensures that $H^{(4)}(z,\cdot)$ is
bounded, in the following sense: there exist 
positive numbers $s_0$ and $\kappa_0$ such that for every $s$ in $[0,s_0]$ and every 
$z$ in $I_t$,
\[
H^{(4)}(z,s)\le 24\kappa_0.
\]
Hence, $\gtt=(\gte,\gti)$ fulfills the first condition to be tempered, with
\[
  k=3, \quad \alpha = 1, \quad \varepsilon_1 = 1, \quad \varepsilon_2
  = 2, \quad \varepsilon_3 = 3, \quad \kappa = \kappa_0, 
\]
and, for every $0\le i\le 2$,
\[
  F_i(z)=H^{(i+1)}(z,0)/H(z,1).
\]
Finally, since $\omega$ is smooth, the functions $F_i$ are bounded on
$I_t$.

\subsection{Proof of Proposition~\ref{prop:disc}}
Recall that, using the random variables $\sse=\ee^{-4\gte}$ and $\ssi
= 1+ 2\ee^{-4\gti}$, the function $ G $ is  
characterized by the fact that, for every measurable bounded function $H$,
$$  
\E \left( H(\sse\ssi):\sse(3 - \ssi ) \leq 2s \right)  =\E \left(
  H(\sse\ssi) G(\sse\ssi,s) \right). 
$$
Here, $\sse$ and $\ssi$ are independent, the distribution of
$\sse$ is uniform on $[0,1]$ and the distribution of $\ssi$ is
discrete with
$$
\Prob( \ssi = y_n ) =  r_n/r.
$$
Thus,
\[
  \sum_nr_n
  \int_0^1
  H(x y_n ) \, \1\{x (3- y_n) \leq 2s\} \, \dd x 
= \sum_nr_n 
  \int_0^1
  H( x y_n  ) \,G( x y_n,s  )  \, \dd x.
\]
The changes of variable $ z=y_nx $ in each integral yield
\[
  \sum_n\frac{r_n}{y_n}
  \int H(z)\,\1\{z\le y_n\}\1\{3z\leq (2s+z)y_n\}\, \dd z 
  =
  \sum_n\frac{r_n}{y_n}
  \int H(z) \,\1\{z\le y_n\}\,G(z,s)\,\dd z.
\]
This must hold for every measurable bounded function $H$, hence
$$
G(z,s) = H(z,s)/H(z,\infty),\quad 
H(z,s)=\sum_n (r_n/y_n) \,\1\{z\le y_n\}\1\{3z\leq (2s+z)y_n\}.
$$
Since $r_n/y_n=n^{-b}-(n+1)^{-b}$ for $n\ge1$, $H(z,s) = n(z,s)^{-b}$
where
\[
n(z,s) 
= 
\inf \{ n \geq 1\,|\, z\le y_n,\,3z \leq (2s+z) y_n \}.
\]
Since $y_n\to3$ when $n\to\infty$, $n(z,s)$ is finite for every $z<3$
and $s>0$.  

For every $z>0$, when $s$ is large enough, namely $s\ge(3-z)/2$,
the condition that $3z \leq (2s+z) y_n$ becomes useless and
$$
n(z,s) 
= 
\inf \{ n \geq 1\,|\, z\le y_n\},
$$ 
hence $n(z,s)$ and $H(z,s)$ are independent of $s$.  If $z\ge1$, this
implies that $n(z,s)$ and $H(z,s)$ are independent of $s\ge1$.  If
$z<1$ and $s\ge1$, the conditions that $z\le y_n$ and $3z \leq (2s+z) y_n$
both hold for every $n\ge1$ hence $n(z,s)=1$ and $H(z,s)=1$. In both
cases, $H(z,\infty)=H(z,1)$.

We are interested in small positive values of $s$. For
every $z<3$, when $s$ is small enough, namely
$s\le(3-z)/2$, the condition $z\le y_n$ becomes useless and
$$
n(z,s) 
= 
\inf \{ n \geq 1\,|\, 3z \leq (2s+z) y_n\},
$$ 
When furthermore $s<z$, $n\ge n(z,s)$ is equivalent to the condition
$$
n^{-a}\le h(s/z),\quad\mbox{with}\quad
h(u)=- \frac14 \ln \left( 1 - \frac{3 u}{1 + 2u}  \right),\ 0\le u<1.
$$
Finally, for every $s<\min\{z,(3-z)/2\}$,
$n(z,s)$ is the unique integer such that
$$
n(z,s) - 1 < h(s/z)^{-1/a} \le n(z,s).
$$
This reads as
$$
  h(u)^{b/a}[1 + h(u)^{1/a}]^{-b}
  <
  H(z,1)\,G(z,s)
  \leq
  h(u)^{b/a},\quad u=s/z.
$$
One sees that the function $h$ is analytic and that $h(u)=(3u/4)+o(u)$
when $u\to0$, hence, 
$$
h(u)^{b/a}=(3u/4)^{b/a}(1+a_1u+a_2u^2+a_3u^3+o(u^3)),
$$
when $u\to0$, for given coefficients $a_1$, $a_2$ and $a_3$. 
Likewise, since $1/a>3$, $h(u)^{1/a}=o(u^3)$ when $u\to0$. This implies that
$$
\left(1 + h(u)^{1/a}\right)^{-b}=1+o(u^3),
$$ 
hence
$$
H(z,1)\,G(z,s)=(3u/4)^{b/a}(1+a_1u+a_2u^2+a_3u^3+o(u^3)).
$$
This yields the first part of Definition~\ref{defi:temp}, with
$$
k=3,\quad
\alpha=b/a,\quad
(\varepsilon_1,\varepsilon_2,\varepsilon_3)=(1,2,3),
$$
and
$$
F_0(z)=(3/4z)^{b/a}/H(z,1),\quad
F_1(z)=a_1F_0(z)/z,\quad
F_1(z)=a_2F_0(z)/z^2.
$$
The remaining step is to get rid of the dependencies over $z$ of our
upper bounds. For 
instance, the reasoning above provides as an error term a multiple of
$$
u^{\alpha+3}/H(z,1)=s^{\alpha+3}/(z^{\alpha+3}H(z,1)),
$$ 
instead of a constant multiple of
$s^{\alpha+3}$. But $\inf I_t>0$, hence the $1/z^{\alpha+3}$
contribution is uniformly bounded. 

As regards $H(z,1)$, we first note
that $H(z,1)=1$ if $z\le1$. If $z\ge1$, elementary computations show
that $H(z,1)\ge c$ if and only if $n(z,1)\le c^{-1/b}$ if and only if
$\exp(-c^{a/b})\ge (z-1)/2$, which is implied by the fact that
$1-c^{a/b}\ge(z-1)/2$, which is equivalent to
the upper bound $c^{a/b}\le(3-z)/2$. Since $\sup I_t<3$, this can be achieved
uniformly over $z$ in $I_t$ and $1/H(z,1)$ is uniformly bounded as well.

Finally, we asked for an expansion valid on $s\le s_0$, for a fixed
$s_0$, and we proved an expansion valid over $s/z\le u_0$, for a fixed
$u_0$. But one can choose $s_0=u_0\inf I_t$. This concludes the proof
that the conditions in the first part of Definition~\ref{defi:temp}
hold. 

We now prove that the second part of Definition~\ref{defi:temp}
holds. Since $\gti$ and $\gte$ are independent, 
for every positive integer $n$,
\[
  Q_n(t) = \Prob \left( \gti \leq 1/n \right) \Prob \left( t \leq \gte
    \leq t + 1/n \right). 
\]
One has
\[
  n \Prob \left( t \leq \gte \leq t + 1/n \right) \to 4 \ee^{-4 t}
  \quad \mbox{when} \quad n \to + \infty, 
\]
and
\[
  \frac1{r(n^{1/a} + 1)^b} \leq \Prob ( \gti \leq 1/n ) \leq \frac3{rn^{b/a}}. 
\]
Since $Q_n(t)$ is bounded from below by a multiple of $1/n^{1+b/a}$, 
the second point of Definition~\ref{defi:temp} holds.

\subsection{Proof of Proposition~\ref{prop:examtemp}}
Recall once again that, using the random variables $\sse=\ee^{-4\gte}$ and $\ssi
= 1+ 2\ee^{-4\gti}$, the function $ G $ is  
characterized by the fact that, for every measurable bounded function $H$,
$$  
\E \left( H(\sse\ssi):\sse(3 - \ssi ) \leq 2s \right)  =\E \left(
  H(\sse\ssi) G(\sse\ssi,s) \right). 
$$
\textbf{Case} $\mathbf{(i)}$. Here, $\sse$ and $\ssi$ are independent,
the distribution of 
$\sse$ is uniform on $[0,1]$ and $\ssi$ is a
continuous random variable with density
\[
  \frac1{4 \theta (s_i-1)} \1\{1+2\ee^{-4 \theta} \leq s_i \leq 3\} 
\]
with respect to the Lebesgue measure. Let $\varpi$ denote the joint
probability density defined as
\[
  \varpi(x,y)=
  \1\{0 \leq x \leq 1\}
  \1\{1+2\ee^{-4 \theta} \leq y \leq 3\} 
  \frac{1}{4 \theta (y-1)}.
\]
Thus, 
\[
  \iint H(xy) \1\{ 3x \leq 2s+xy \}  \varpi(x,y) \dd x  \dd y   
= \iint H(xy) G(xy,s) \varpi(x,y) \dd x  \dd y.
\]
The change of variable $z=xy$ yields
\[
  \iint H(z) \1\{ 3x\leq 2s+z \} \varpi \left( x,z/x \right) \dd z \dd x  / x
=
\iint H(z) G(z,s) \varpi \left(x,z/x \right) \dd z \dd x  / x.
\]
This must hold for every measurable bounded function $H$, one can choose
\[
  G(z,s) = H(z,s)/H(z,\infty),
\]
with
\[
  H(z,s) = \int \ \1\{3x \leq 2s + z\} \1\{0 \leq x \leq 1\} \1 \{ 1 +
  2 \ee^{- 4 \theta} \leq z/x \leq 3\} \dd x /(z-x).
\]
Finally, for every $s \geq 0$ and $z $ in $[0,3]$,
\[
  G(z,s) = H(z,s)/H(z,1+\ee^{-4\theta}),
\]
where
\[
  H(z,s) = \int_{m(0,z)}^{m(s,z)} \frac{\dd x}{z-x},
  \quad \mbox{with } m(s,z) = \min\{1,z/(1 + 2 \ee^{-4 \theta}), (2s+z)/3\}.
\]
Hence, $m(0,z)=z/3$ and, for small positive values of $s$,
$m(s,z)=m(0,z)+2s/3$.  When $0 \le z \le 1 + 2 \ee^{- 4 \theta}$,
$m(s,z)\to m(\infty,z)=z/(1 + 2 \ee^{-4 \theta})$ when $s\to\infty$ and
this limit is reached for $s=\frac{1 + \ee^{-4 \theta}}{1 + 2 \ee^{-4
\theta}}z$.  When $1+ 2 \ee^{-4 \theta} \le z \le 3 $, $m(s,z)\to
m(\infty,z)=1$ when $s\to\infty$ and this limit is reached for
$s=(3-z)/2$.  In both cases, $m(\infty,z)=m(1+\ee^{-4\theta},z)$ hence
$H(z,\infty)=H(z,1+\ee^{-4\theta})$.

For every fixed $0 \leq z \leq 1 + 2\ee^{-4 \theta}$ and every
$\displaystyle 0 \leq 
s \leq \frac{1 + \ee^{-4 \theta}}{1 + 2 \ee^{-4
\theta}}z$,
\[
  H(z,s) = \log \left( \frac{z}{z-s} \right).
\]
For every fixed $ 1 + 2\ee^{-4 \theta} \leq z \leq 3 $ and every
$\displaystyle 0 \leq 
s \leq \frac{3 - z}{2}$,
\[
  H(z,s) = \log \left( \frac{z}{z-s} \right).
\]
Hence, there exists a positive $s_0$ such that for every $z $ in $I_t$
and every  $s $ in $[0,s_0]$, 
\[
  H(z,s) = \log \left( \frac{z}{z-s} \right) = \log \left(1 -
    \frac{s}{z} \right).
\]
Such a function has a Taylor expansion around $s=0$ with uniformly
bounded coefficient over $z $ in $I_t$. 
Hence, $\gtt=(\gte,\gti)$ fulfills the first condition to be tempered.

We now prove that the second part of Definition~\ref{defi:temp}
holds. Since $\gti$ and $\gte$ are independent, 
for every positive integer $n$,
\[
  Q_n(t) = \Prob \left( \gti \leq 1/n \right) \Prob \left( t \leq \gte
    \leq t + 1/n \right). 
\]
One has
\[
  n \Prob \left( t \leq \gte \leq t + 1/n \right) \to 4 \ee^{-4 t}
  \quad \mbox{when} \quad n \to + \infty, 
\]
and
\[
   \Prob ( \gti \leq 1/n ) =  \frac1{\theta n}, \quad \mbox{when} \quad
   n  \mbox{ is large enough}.
\]
Since $Q_n(t)$ is bounded from below by a multiple of $1/n^2$, 
the second point of definition~\ref{defi:temp} holds.

\textbf{Case} $\mathbf{(ii)}$. Here, $\sse$ and $\ssi$ are
independent, the distribution of 
$\sse$ is uniform on $[0,1]$ and $\ssi$ is a
continuous random variable with density
\[
  \frac{\theta}{4 (s_i-1)} \left[ -\frac14 \log \left( \frac{s_i-1}{2}
    \right)  \right]^{\theta - 1}\1\{1+2\ee^{-4} \leq s_i < 3\}  
\]
with respect to the Lebesgue measure. 
One can choose
\[
  G(z,s) = H(z,s)/H(z,\infty),
\]
where
\[
  H(z,s) = \int_{m(0,z)}^{m(s,z)} \   \left[
    \frac{-1}{4} \log \left( 
      \frac{z-x}{2x} \right) \right]^{\theta-1} \frac{\dd x}{z-x},
\]
with
\[
  m(s,z) = \min\{1, z/(1+ 2 \ee^{-4}), (2s+z)/3\}.
\]
Hence, $m(0,z)=z/3$ and, for small positive values of $s$,
$m(s,z)=m(0,z)+2s/3$.  When $0 \le z \le 1 + 2 \ee^{- 4}$,
$m(s,z)\to m(\infty,z)=z/(1 + 2 \ee^{-4 })$ when $s\to\infty$ and
this limit is reached for $s=\frac{1 + \ee^{-4 }}{1 + 2 \ee^{-4
}}z$.  When $1+ 2 \ee^{-4 } \le z \le 3 $, $m(s,z)\to
m(\infty,z)=1$ when $s\to\infty$ and this limit is reached for
$s=(3-z)/2$.  In both cases, $m(\infty,z)=m(1+\ee^{-4},z)$ hence
$H(z,\infty)=H(z,1+\ee^{-4})$.

Hence, there exists a positive $s_0$ such that for every $z $ in $I_t$
and every  $s $ in $[0,s_0]$, 
\[
  H(z,s) = \int_{0}^s \frac{1}{(z-x)} \left[ \frac{-1}{4} \log \left(
      1 - \frac{3x}{2x+z} \right) \right]^{\theta-1} \dd x.
\]
Such a function has a Taylor expansion around $s=0$ with uniformly
bounded coefficient over $z $ in $I_t$. For instance, when $\theta=1/2$, 
\[
  H(z,s) = \frac{4}{\sqrt{3z}}\sqrt{s} + \frac{5}{(3z)^{3/2}} s^{3/2}
  + \frac{9 \sqrt{3} }{40 z^{5/2} } s^{5/2} + O(s^{7/2}), 
\]
where $O(s^{7/2})$ is uniformly bounded over  $z $ in $I_t$. Hence,
$\gtt=(\gte,\gti)$ fulfills the first condition to be tempered. 

We now prove that the second part of Definition~\ref{defi:temp}
holds. Since $\gti$ and $\gte$ are independent, 
for every positive integer $n$,
\[
  Q_n(t) = \Prob \left( \gti \leq 1/n \right) \Prob \left( t \leq \gte
    \leq t + 1/n \right). 
\]
One has
\[
  n \Prob \left( t \leq \gte \leq t + 1/n \right) \to 4 \ee^{-4 t}
  \quad \mbox{when} \quad n \to + \infty, 
\]
and
\[
   \Prob ( \gti \leq 1/n ) =  \frac1{n^{\theta}}, \quad \mbox{when} \quad
   n  \mbox{ is large enough}.
\]
Since $Q_n(t)$ is bounded from below by a multiple of $1/n^{1+\theta}$, 
the second point of definition~\ref{defi:temp} holds.

\textbf{Case} $\mathbf{(iii)}$. Here, $\sse$ and $\ssi$ are
independent, the distribution of 
$\sse$ is uniform on $[0,1]$ and $\ssi$ is a
continuous random variable with density
\[
  -\frac{1}{4 (s_i-1)} \log \left[ -\frac14 \log \left( \frac{s_i-1}{2}
    \right)  \right]\1\{1+2\ee^{-4} \leq s_i < 3\}  
\]
with respect to the Lebesgue measure.

One can choose
\[
  G(z,s) = H(z,s)/H(z,\infty),
\]
where
\[
  H(z,s) = \int_{m(0,z)}^{m(s,z)} \   \log \left[
    \frac{-1}{4} \log \left( 
      \frac{z-x}{2x} \right) \right] \frac{\dd x}{z-x},
\]
with
\[
  m(s,z) = \min\{1, z/(1+ 2 \ee^{-4}), (2s+z)/3\}.
\]

Hence, there exists a positive $s_0$ such that for every $z $ in $I_t$
and every  $s $ in $[0,s_0]$, 
\[
  H(z,s) = -\int_{0}^s \frac{1}{(z-x)} \log \left[ \frac{-1}{4} \log \left(
      1 - \frac{3x}{2x+z} \right) \right] \dd x.
\]
The Taylor expansion around zero of $H(z,s)$ reads as
\[
z\,H(z,s)=\left(1 - \log(3/(4z))\right) s - s\,\log(s)+o(s\,\log(s)),
\]
hence $\gtt=(\gte,\gti)$ does not fulfill the first condition to be tempered.

\textbf{Case} $\mathbf{(iv)}$. Here, $\sse$ and $\ssi$ are
independent, the distribution of 
$\sse$ is uniform on $[0,1]$ and $\ssi$ is a
continuous random variable with density
\[
  \frac{1}{16 (s_i-1)} \log \left( \frac{s_i-1}{2}
    \right)  \log \left[ -\frac14 \log \left( \frac{s_i-1}{2}
    \right)  \right]\1\{1+2\ee^{-4} \leq s_i < 3\}  
\]
with respect to the Lebesgue measure.

One can choose
\[
  G(z,s) = H(z,s)/H(z,\infty),
\]
where
\[
  H(z,s) = \int_{m(0,z)}^{m(s,z)} \ \log \left( 
      \frac{z-x}{2x} \right) \  \log \left[
    \frac{-1}{4} \log \left( 
      \frac{z-x}{2x} \right) \right] \frac{\dd x}{z-x},
\]
with
\[
  m(s,z) = \min\{1, z/(1+ 2 \ee^{-4}), (2s+z)/3\}.
\]

Hence, there exists a positive $s_0$ such that for every $z $ in $I_t$
and every  $s $ in $[0,s_0]$, 
\[
  H(z,s) = -\int_{0}^s \frac{1}{(z-x)} \log \left(
      1 - \frac{3x}{2x+z} \right) \log \left[ \frac{-1}{4} \log \left(
      1 - \frac{3x}{2x+z} \right) \right] \dd x.
\]
The Taylor expansion around zero of $H(z,s)$ reads as
\[
2z^2\,H(z,s)=\left(3/2 -3 \log(3) + 3\log(z)  + 6 \log(2)\right) s^2
- 3s^2\,\log(s)+o(s^2\,\log(s)),
\]
hence $\gtt=(\gte,\gti)$ does not fulfill the first condition to be tempered.


\section{Proof of Proposition~\ref{prop:main}} \label{appe:main}
\begin{nota}
Recall that $\Gamma$ denotes the Gamma function defined for every positive
number $x$ by
\[
  \Gamma(x) = \int_0^{+ \infty} t^{x-1}\ee^{-t} \dd t.
\]
For every real number $t$, let $[t]$ denote the integer part of $t$,
that is, the largest integer not greater than $t$, and let $\{t\}$
denote the fractional part of $t$, hence $t=\{t\}+[t]$, $[t]$ is an integer and $\{t\}$ belongs to the interval $[0,1)$.
\end{nota} 

For fixed values of the coefficients $\alpha$, $\gamma_i$ and
$\varepsilon_i$, introduce, for every $t>0$,  
$$
  M^\pm_t = \int_0^1 t v^{t-1} F_{\pm}(v) \,\dd v, 
  \quad \mbox{where } 
  F_{\pm}(v) = \sum_{i=0}^{n-1} \gamma_i ( 1 - v
  )^{\alpha+\varepsilon_i} \pm \gamma_n (1-v)^{\alpha + \varepsilon_n
  }. 
$$
Hence, 
$$
  M_t = \int_0^1 t v^{t-1} \Prob(V \geq v)\, \dd v = M^\pm_t+ \int_0^1
  t v^{t-1} [ \Prob(V\geq v) - F_\pm(v) ] \,\dd v, 
$$
and
$$
M^\pm_t =t B(t,\alpha +1) \left( \sum_{i=1}^{n-1} \gamma_i
  \Lambda(\varepsilon_i,t) P(\varepsilon_i,t) \pm \gamma_n
  \Lambda(\varepsilon_n,t) P(\varepsilon_n,t) \right)   , 
$$
where
$$
  \Lambda(\varepsilon,t)= \frac{\Gamma(\{t\} + \alpha + 1  )}{
    \Gamma(\{t\} + \alpha + \varepsilon + 1  ) }, 
  \quad
  P(\varepsilon,t) = \prod_{\ell=1}^{[t]+1} \left( 1 -
    \frac{\varepsilon}{\alpha+\varepsilon+\{t\}+\ell} \right), 
$$
and $B$ denotes the beta function
$$
  B(x,y) = \frac{\Gamma(x)\Gamma(y)}{\Gamma(x+y)}.
$$
From the control of the distribution of $V$,
$$
  M^-_t - \gamma v_0^t \leq M_t \leq M^+_t + v_0^t
\quad \mbox{where } \gamma = \sum_{i=0}^{n} |\gamma_i|.
$$
Combining this with the general expression of $M^\pm_t$ given above, one gets
$$ 
\frac{M_{t+1}}{M_t} 
  \leq   \frac{
   (t+1)B(t+1,\alpha+1)\chi_+(t+1) + v_0^{t+1}
  }
  {
  t B(t,\alpha + 1) \chi_-(t)  - \gamma v_0^{t}
  },
$$
where
$$
  \chi_\pm(t)=\sum_{i=0}^{n-1} \gamma_i \Lambda(\varepsilon_i,t)
  P(\varepsilon_i,t) \pm \gamma_n \Lambda(\varepsilon_n,t)
  P(\varepsilon_n,t). 
$$
Using the fact that
\[ 
  \frac{(t+1) B(t+1,\alpha +1)}{t B(t,\alpha +1)} = \frac{t+1}{t+\alpha+1},
\]
and that
\[
  t B(t,\alpha +1) Q_\alpha(t)\ge1, \quad \mbox{where } Q_\alpha(t) =
  \frac{(t + \alpha)(t+\alpha-1) \dots (t+\{\alpha\}) }{ \Gamma(\alpha
    + 1 )}, 
\]
one sees that
$$
 \frac{M_{t+1}}{M_t} 
  \leq 
  \frac{t+1}{t+\alpha+1}
  \frac{
  \gamma_0 + \chi_+(t+1)
  +
  Q_\alpha(t+1)v_0^{t+1}
  }
  {
  \gamma_0 + \chi_-(t)
  - 
  \gamma Q_\alpha(t)v_0^t
  }.
$$
Furthermore,
\[
  \frac{
  \gamma_0 + \chi_+(t+1)
  +
  Q_\alpha(t+1)v_0^{t+1}
  }
  {
  \gamma_0 + \chi_-(t)
  - 
  Q_\alpha(t)v_0^t
  }
  =
  1 + \frac{\chi_+(t+1) - \chi_-(t) + \kappa(t) v_0^t
  }
  {\gamma_0 + \chi_-(t)
  - 
  \gamma Q_\alpha(t)v_0^t}.
\]
where $\kappa(t) = v_0 Q_\alpha(t+1) + \gamma Q_\alpha(t)$ is a
polynomial function in $t$. 

From Lemma~\ref{lemm:chi} below, there exists a positive number $C$ which
depend on the exponents $\alpha$ and $\varepsilon_i$, $0\le i\le n$, only, such that
$$
\chi_+(t+1) - \chi_-(t)  \leq [2 \gamma_n + \varepsilon_n \gamma ] C t^{-\beta},
\quad
\chi_-(t) \geq - C \gamma t^{-\varepsilon_1}.
$$
where
$$ 
\beta = \min \{ \varepsilon_n,1+\varepsilon_1 \},\qquad
1<\beta\le2.
$$
Combining these estimates on $\chi_+(t+1)$ and $\chi_-(t)$, one sees
 that there exists finite continuous functions $\theta_1$ and
$A$ of the exponents $\gamma_i$, $\alpha$, and $\varepsilon_i$, such
that, for every $t\geq \theta_1$,
\[
 R_t\ge \alpha/t -A/t^{\beta}.
\]
Since $\beta>1$, there exists $\theta_2$ such that
$2A\,t\le\alpha\, t^{\beta}$ for every $t\ge\theta_2$.
Choosing finally $\theta=\max(\theta_1,\theta_2)$ yields Proposition~\ref{prop:main}.


\begin{lemm} \label{lemm:chi}
Let $\beta = \min \{ \varepsilon_n,1+\varepsilon_1 \}$.
There exists a positive number $C$, which depends on the exponents $\alpha$ and 
$\varepsilon_i$ only, such that
$$  \chi_+(t+1) - \chi_-(t)  \leq [2 \gamma_n + \varepsilon_n \gamma ]
C t^{-\beta}, \quad 
\chi_-(t) \geq - C \gamma t^{-\varepsilon_1}.
$$
\end{lemm}
\begin{proof}[Proof of Lemma~\ref{lemm:chi}]
For every real number $t\geq 1$ and every $1\le i\le n$,
\[
  \ee^{-S(\varepsilon_i,t) - T(\varepsilon_i,t)} \leq
  P(\varepsilon_i,t) \leq \ee^{-S(\varepsilon_i,t)}, 
\]
where
\[
  S(\varepsilon,t) = \sum_{\ell=1}^{[t]+1}
  \frac{\varepsilon}{\alpha+\varepsilon + \{t\} + \ell}  
  \quad \mbox{and} \quad 
  T(\varepsilon,t) = \sum_{\ell=1}^{[t]+1}
  \frac{\varepsilon^2}{(\alpha+\varepsilon + \{t\} + \ell)^2}. 
\]
Thus, there exists two positive real numbers $C_i^-$ and $C_i^+$ such
that for every real number $t \geq 1$, 
$C_i^- \leq t^{\varepsilon_i} P(\varepsilon_i,t) \leq C_i^+$,
and one can choose $C_i^+=(\alpha+\varepsilon_i+3)^{\varepsilon_i}$.

Let $C=\max\{C_i^+\,;\,1\le i\le n\}$.
Using the two relations
\[
  P(\varepsilon_i,t) - P(\varepsilon_i,t+1) = P(\varepsilon_i,t)
  \frac{\varepsilon_i}{\alpha + \varepsilon_i + t + 2}, 
\]
and
\[
  P(\varepsilon_n,t) + P(\varepsilon_n,t+1) =
  P(\varepsilon_n,t)\left(2- \frac{\varepsilon_n}{\alpha +
      \varepsilon_n + t + 2} \right), 
\]
one sees that
\[
  \chi_+(t+1) - \chi_-(t)  = 2 \gamma_n \Lambda(\varepsilon_n,t)
  P(\varepsilon_n,t) - \sum_{i=1}^{n} \gamma_i
  \Lambda(\varepsilon_i,t) P(\varepsilon_i,t)
  \frac{\varepsilon_i}{\alpha+\varepsilon_i+t+2}. 
\]
For every $1\le i\le n$, the function $\Lambda(\varepsilon_i,\cdot)$
is positive and bounded by $1$. Hence,
\begin{align*}
\chi_+(t+1) - \chi_-(t)  
  &\leq 2 \gamma_n  P(\varepsilon_n,t) + \sum_{i=1}^{n} |\gamma_i|
  P(\varepsilon_i,t) \frac{\varepsilon_i}{\alpha+\varepsilon_i+t+2}\\ 
  &\leq C\,\left(2 \gamma_n t^{-\varepsilon_n} + \gamma \varepsilon_n
    t^{-(1+\varepsilon_1)}\right), 
\end{align*}
and the first inequality in the statement of the lemma holds.
The same kind of estimates  yields
\[
  \chi_-(t) \geq -\sum_{i=0}^{n-1} |\gamma_i| \Lambda(\varepsilon_i,t)
  P(\varepsilon_i,t) - \gamma_n \Lambda(\varepsilon_n,t)
  P(\varepsilon_n,t), 
\]
hence the second inequality holds.
This concludes the proof of Lemma~\ref{lemm:chi}.
\end{proof}

\end{document}